\documentclass{amsart}
\usepackage{amssymb}
\usepackage{fullpage}
\usepackage{stmaryrd}
\usepackage{mathrsfs}

\newtheorem{theorem}{Theorem}[section]

\newtheorem{prop}[theorem]{Proposition}

\theoremstyle{definition}

\theoremstyle{remark}

\numberwithin{equation}{section}

\newcommand{\bC}{{\mathbf{C}}}

\newcommand{\bF}{{\mathbf{F}}}

\newcommand{\Irr}{{\operatorname{Irr}}}
\newcommand{\cl}{{\operatorname{cl}}}

\newcommand{\Hall}{{\operatorname{Hall}}}

\let\nor=\triangleleft

\begin{document}

\title{Orbits of finite solvable groups on characters}

\author{Thomas Michael Keller}
\address{Department of Mathematics, Texas State University at San Marcos, 601 University Drive, San Marcos, TX 78666, USA.}
\author{YONG YANG}
\address{Department of Mathematics, University of Wisconsin at Parkside, 900 Wood Road, Kenosha, WI 53144, USA.}
\makeatletter
\email{keller@txstate.edu, yangy@uwp.edu}
\makeatother
\thanks{Part of this work was done while the first author was on sabbatical leave and
was visiting the University of Wisconsin-Parkside. He thanks the Mathematics Department
there for its hospitality.}

\date{}



\begin{abstract}
We prove that if a solvable group $A$ acts coprimely on a solvable group $G$, then $A$ has a ``large" orbit in its corresponding action
on the set of ordinary complex irreducible characters of $G$. This extends (at the cost of a weaker bound) a 2005 result of A. Moret\'o
who obtained such a bound in case that $A$ is a $p$-group.
\end{abstract}

\maketitle
\Large
\section{Introduction} \label{sec:introduction8}

The main purpose of this paper is to generalize a 2005 result of A. Moret\'o on orbits in a certain group action. While there are many results on the orbits in the action of a group acting via automorphisms on some other group (of which the action of linear groups on their natural modules is a particularly prominent example), Moret\'o's result is noteworthy in that it is one of the very few dealing with the action of a group on the set of irreducible characters of another group. \\

More precisely, let the finite group $A$ act (via automorphisms)
on the finite group $G$. Such an action induces an action of $A$ on the set $\Irr(G)$ in an obvious way (where $\Irr(G)$ denotes the set
of complex irreducible characters of $G$). When $G$ is elementary abelian, we are back to studying linear group actions and all the known results apply. But for nonabelian $G$ not much is known about this interesting action. Note that when $(|A|,|G|)=1$, then it is well-known that the orbit sizes of $A$ on $\Irr(G)$ are the same as the orbit sizes in the natural action of $A$ on the conjugacy classes of $G$, and this latter action was of some importance in \cite{KELL}, where some specialized results on this action were obtained. But apart from these we are aware only of two major results on the action of $A$ on $\Irr(G)$. \\

The first such result is due to D. Gluck ~\cite{GLUCK2}. He proved that when $A$ is abelian and $G$ is solvable, then there always exists an ``arithmetically large" orbit on $\Irr(G)$ (i.e., an orbit whose size is divisible by ``many" different primes). \\

The second result is the 2005 result \cite{MO1} by Moret\'o mentioned above. It proves the existence of a ``large" orbit on $\Irr(G)$ in case that $A$ is a $p$-group for some prime $p$ and $G$ is solvable such that $(|A|,|G|)=1$. \\

In this paper we take the second result to the next level and establish the existence of a large orbit on $\Irr(G)$ in case that $A$ is solvable and $G$ is solvable such that $(|A|,|G|)=1$. More precisely, our main result is the following.\\

{\bf Theorem A.} {\it Let $A$ and $G$ be finite solvable groups such that $A$ that acts faithfully and coprimely on $G$. Let $b$ be an integer such that $|A : \bC_A(\chi)| \leq b$ for all $\chi \in \Irr(G)$. Then $|A| \leq b^{49}$.}\\

We make a few observations. First, it has already been observed in \cite{MO1} that the coprimeness assumption cannot be omitted.\\

Second, not surprisinglyly, our bound is weaker than the bound obtained in  \cite{MO1}. While in \cite{MO1}, when $A$ is a $p$-group, the existence of an orbit of size roughly $|A|^{\frac{1}{19}}$ on $\Irr(G)$ is proved, in the more general situation of Theorem A we only get an orbit of size about $|A|^{\frac{1}{49}}$. Both bounds, however, are far from best possible anyway and thus the results are more of a qualitative nature. The true bound is probably close to $|A|^{\frac{1}{2}}$ in both cases.\\

Third, if we assume $|AG|$ to contain no small primes, the bounds one can get tend to be much better. In \cite{MO1} it is observed in passing that if odd order is assumed, then the bound will get much better. And here we will explicitly establish a better bound in case that $|G|$ is not divisible by 6 (see Theorem ~\ref{coprimeactionchar23} below).\\

Our proof of Theorem A is largely based on the ideas introduced in \cite{MO1} and extends them to the more general hypothesis of Theorem A. In particular, our proof does not use Moret\'o's result, but rather reproves it (with a weaker bound) as a special case. To keep the bound from getting too large, we also make use of a recent strenghtening in the solvable case of a result by M. Aschbacher and R. Guralnick \cite{ASCHGURA} on the size of $|G/G'|$ of a linear solvable group; see ~\ref{orbitabelianquotient}.

\section{The abelian quotient of linear groups} \label{sec:orbittheorem}

We first recall a result due to Aschbacher and Guralnick.

\begin{prop}\label{prop1}
Let $G$ be a finite solvable group that acts faithfully and completely reducibly on a finite vector space $V$. Then $|G:G'| \leq |V|$.
\end{prop}
\begin{proof}
This is a special case of the much more general \cite[Theorem 3]{ASCHGURA}.
\end{proof}

Next we provide a recent strengthening of Proposition ~\ref{prop1} by the authors.

\begin{theorem}\label{orbitabelianquotient}
Let $G$ be a finite solvable group that acts faithfully and completely reducibly on a finite vector space $V$, and let $B$ be the size of the largest orbit of $G$ on $V$. Then $|G:G'| \leq B$.
\end{theorem}
\begin{proof}
This will appear in \cite{KELLYANG}.
\end{proof}

\section{Main Theorem} \label{sec:maintheorem}

The following result is an extension of ~\cite[Lemma 2.1]{MO1}.

\begin{theorem} \label{pigroupaction}
Assume that a solvable $\pi$-group $A$ acts faithfully on a solvable $\pi'$-group $G$. Let $b$ be an integer such that $|A : \bC_A(\chi)| \leq b$ for all $\chi \in \Irr(G)$. Let $\Gamma = AG$ be the semidirect product. Let $K_{i+1}=\bF_{i+1}(\Gamma)/\bF_i(\Gamma)$ and let $K_{i+1, \pi}$ be the Hall $\pi$-subgroup of $K_{i+1}$ for all $i \geq 1$. Let $K_i/\Phi(\Gamma /\bF_{i-1}(\Gamma))=V_{i1}+V_{i2}$ where $V_{i1}$ is the $\pi$ part of $K_i/\Phi(\Gamma/\bF_{i-1}(\Gamma))$ and $V_{i2}$ is the $\pi'$ part of $K_i/\Phi(\Gamma/\bF_{i-1}(\Gamma))$ for all $i \geq 1$. Let $K \nor \Gamma$ such that $\bF_i(\Gamma) \nor K$. Let $L_{i+1, \pi}=K_{i+1, \pi} \cap K$. We have that $|\bC_{L_{i+1, \pi}}(V_{i1})| \leq b^2$, and $|\bC_{L_{i+1, \pi}}(V_{i1})| \leq b$ if $L_{i+1, \pi}$ is abelian. The order of the maximum abelian quotient of $\bC_{L_{i+1, \pi}}(V_{i1})$ is less than or equal to $b$ for all $i \geq 1$. 
\end{theorem}
\begin{proof}

We know that $L_{i+1, \pi}$ acts faithfully and completely reducibly on $V_{i1}+V_{i2}$. Clearly $\bC_{L_{i+1, \pi}}(V_{i1})$ acts faithfully and completely reducibly on $V_{i2}$ and $L_{i+1, \pi}/\bC_{L_{i+1, \pi}}(V_{i1})$ acts faithfully and completely reducibly on $V_{i1}$.


Let $L$ be the pre-image of $\bC_{L_{i+1, \pi}}(V_{i1})$ in $(\Gamma/\bF_{i-1}(\Gamma)/\Phi(\Gamma/\bF_{i-1}(\Gamma)))/V_{i1}$. Write $L = Q V_{i2}$, where $Q \in \Hall_{\pi}(L)$. We have to prove that $|Q| \leq b^{2}$. Clearly $\bF(L)=V_{i2}$ and $\Phi(L)=1$. We know by a theorem of Brauer that $Q$ acts faithfully on $\Irr(V_{i2})$. Replacing $A$ by a conjugate, if necessary, we may assume that $Q \leq A$. It follows from our hypothesis that $|Q : \bC_Q(\chi)| \leq b$ for all $\chi \in \Irr(G)$. 

Now, let $\lambda \in \Irr(V_{i2})$. By ~\cite[Theorem 13.28]{IMIB}, there exists $\chi \in \Irr(G)$ lying over $\lambda$ that is $\bC_Q(\lambda)$-invariant. We claim that $\bC_Q(\chi) = \bC_Q(\lambda)$. It is clear that $|Q : \bC_Q(\lambda)|$ divides the degree of any character of $L$ lying over $\lambda$. Therefore, $|Q : \bC_Q(\lambda)|$ divides the degree of any character of $QG$ lying over $\lambda$ since the pre-image of $L$ in $\Gamma$ is normal in $\Gamma$. Now, ~\cite[Corollary 8.16]{IMIB} and Clifford's correspondence ~\cite[Theorem 6.11]{IMIB} yield that there exist $\psi \in \Irr(QG)$ lying over $\chi$, whence over $\lambda$, such that $\psi(1)_{\pi} = |Q: \bC_Q(\chi)|$. It follows that $\bC_Q(\chi) = \bC_Q(\lambda)$, as desired. In particular, $|Q: \bC_Q(\lambda)| \leq b$. We deduce that for all $\lambda \in \Irr(V_{i2})$, $|Q: \bC_Q(\lambda)| \leq b$. Now, ~\cite{IMI2}, for instance, implies that $|Q| \leq b^2$. Also, if $Q$ is abelian, then $|Q| \leq b$.  The order of the maximum abelian quotient of $Q$ is less than or equal to $b$ by Theorem ~\ref{orbitabelianquotient}. This completes the proof of the theorem.
\end{proof}

Now we are ready to prove Theorem A, which we restate.\\

\begin{theorem} \label{coprimeactionchar}
Let $A$ be a solvable $\pi$-group that acts faithfully on a solvable $\pi'$-group $G$. Let $b$ be an integer such that $|A : \bC_A(\chi)| \leq b$ for all $\chi \in \Irr(G)$. Then $|A| \leq b^{49}$.
\end{theorem}
\begin{proof}
Let $\Gamma = AG$ be the semidirect product of $A$ and $G$. By Gaschutz's theorem, $\Gamma/\bF(\Gamma)$ acts faithfully and completely reducibly on $\Irr(\bF(\Gamma)/\Phi(\Gamma))$. It follows from ~\cite[Theorem 3.3]{YY1} that there exists $\lambda \in \Irr(\bF(\Gamma)/\Phi(\Gamma))$ such that $T = \bC_{\Gamma}(\lambda) \leq \bF_8(\Gamma)$.

Let $K_{i+1}=\bF_{i+1}(\Gamma)/\bF_i(\Gamma)$ and let $K_{i+1, \pi}$ be the Hall $\pi$-subgroup of $K_{i+1}$ for all $i \geq 1$.

We know that $K_{i+1, \pi}$ acts faithfully and completely reducibly on $K_i/\Phi(\Gamma/\bF_{i-1}(\Gamma))$. It is clear that we may write $K_i/\Phi(\Gamma/\bF_{i-1}(\Gamma))=V_{i1}+V_{i2}$ where $V_{i1}$ is the $\pi$ part of $K_i/\Phi(\Gamma/\bF_{i-1}(\Gamma))$ and $V_{i2}$ is the $\pi'$ part of $K_i/\Phi(\Gamma/\bF_{i-1}(\Gamma))$ for all $i \geq 1$. Clearly $\bC_{K_{i+1, \pi}}(V_{i1})$ acts faithfully and completely reducibly on $V_{i2}$. Thus $|\bC_{K_{i+1, \pi}}(V_{i1})| \leq b^2$ and the order of the maximum abelian quotient of $\bC_{K_{i+1, \pi}}(V_{i1})$ is less than or equal to $b$ by Theorem ~\ref{pigroupaction}. Also $K_{i+1, \pi}/\bC_{K_{i+1, \pi}}(V_{i1})$ acts faithfully and completely reducibly on $V_{i1}$. Since $K_{i+1, \pi}/\bC_{K_{i+1, \pi}}(V_{i1})$ is nilpotent, $|K_{i+1, \pi}/\bC_{K_{i+1, \pi}}(V_{i1})| \leq |V_{i1}|^{\beta}/2$ where $\beta=\log(32)/\log (9)$ by ~\cite[Theorem 3.3]{MAWOLF}. Also the order of the maximum abelian quotient of $K_{i+1, \pi}/\bC_{K_{i+1, \pi}}(V_{i1})$ is bounded above by $|V_{i1}|$ by Proposition ~\ref{prop1}.

Thus we have the following, $|K_{2,\pi}| \leq b^2$ and the order of the maximum abelian quotient of $K_{2,\pi}$ is bounded above by $b$.

$|K_{3,\pi}| \leq |\bC_{K_{3, \pi}}(V_{21})| \cdot |K_{3, \pi}/\bC_{K_{3, \pi}}(V_{21})| \leq b^2 \cdot b^{\beta}$ and the order of the maximum abelian quotient of $K_{3,\pi}$ is bounded above by $b \cdot b=b^2$.

$|K_{4,\pi}| \leq |\bC_{K_{4, \pi}}(V_{31})| \cdot |K_{4, \pi}/\bC_{K_{4, \pi}}(V_{31})| \leq b^2 \cdot b^{2\beta}$ and the order of the maximum abelian quotient of $K_{4,\pi}$ is bounded above by $b \cdot b^2=b^3$.

$|K_{5,\pi}| \leq |\bC_{K_{5, \pi}}(V_{41})| \cdot |K_{5, \pi}/\bC_{K_{5, \pi}}(V_{41})| \leq b^2 \cdot b^{3\beta}$ and the order of the maximum abelian quotient of $K_{5,\pi}$ is bounded above by $b \cdot b^3=b^4$.

$|K_{6,\pi}| \leq |\bC_{K_{6, \pi}}(V_{51})| \cdot |K_{6, \pi}/\bC_{K_{6, \pi}}(V_{51})| \leq b^2 \cdot b^{4\beta}$ and the order of the maximum abelian quotient of $K_{6,\pi}$ is bounded above by $b \cdot b^4=b^5$.

$|K_{7,\pi}| \leq |\bC_{K_{7, \pi}}(V_{61})| \cdot |K_{7, \pi}/\bC_{K_{7, \pi}}(V_{61})| \leq b^2 \cdot b^{5\beta}$ and the order of the maximum abelian quotient of $K_{7,\pi}$ is bounded above by $b \cdot b^5=b^6$.

$|K_{8,\pi}| \leq |\bC_{K_{8, \pi}}(V_{71})| \cdot |K_{8, \pi}/\bC_{K_{8, \pi}}(V_{71})| \leq b^2 \cdot b^{6\beta}$.



Next, we show that $|\Gamma : T|_{\pi} \leq b$.

Let $\chi$ be any irreducible character of $G$ lying over $\lambda$. Then every irreducible character of $\Gamma$ that lies over $\chi$ also lies over $\lambda$ and hence has degree divisible by $|\Gamma : T|$. But $\chi$ extends to its stabilizer in $\Gamma$ and thus some irreducible character of $\Gamma$ lying over $\chi$ has degree $\chi(1) |A : C_A(\chi)|$. The $\pi$-part of $|\Gamma : T|$, therefore, divides $|A : \bC_A(\chi)|$, which is at most $b$. 

This gives that $|A| \leq b^2 \cdot b^2 \cdot b^{\beta} \cdot b^2 \cdot b^{2\beta} \cdot b^2 \cdot b^{3\beta} \cdot b^2 \cdot b^{4\beta} \cdot b^2 \cdot b^{5\beta} \cdot b^2 \cdot b^{6\beta} \cdot b=b^{15+21 \beta} \leq b^{48.124}$
\end{proof}

\begin{theorem} \label{coprimeactionchar23}
Let $A$ be a solvable $\pi$-group that acts faithfully on a solvable $\pi'$-group $G$. Assume that $2,3 \not\in \pi$. Let $b$ be an integer such that $|A : \bC_A(\chi)| \leq b$ for all $\chi \in \Irr(G)$. Then $|A| \leq b^{4}$.
\end{theorem}
\begin{proof}
Let $\Gamma = AG$ be the semidirect product of $A$ and $G$. By Gaschutz's theorem, $\Gamma/\bF(\Gamma)$ acts faithfully and completely reducibly on $\Irr(\bF(\Gamma)/\Phi(\Gamma))$. It follows from ~\cite[Theorem 3.2]{YY11} that there exists $\lambda \in \Irr(\bF(\Gamma)/\Phi(\Gamma))$ and $K \nor \Gamma$ such that $T = \bC_{\Gamma}(\lambda) \subseteq K$, $\bF(\Gamma) \subseteq K \subseteq \bF_3(\Gamma)$. The $\pi$-subgroup of $K \bF_2(\Gamma)/\bF_2(\Gamma)$ and the $\pi$-subgroup of $(K \cap \bF_2(\Gamma))/ \bF(\Gamma)$ are abelian.

Let $K_{i+1}=\bF_{i+1}(\Gamma)/\bF_i(\Gamma)$ and let $K_{i+1, \pi}$ be the Hall $\pi$-subgroup of $K_{i+1}$ for all $i \geq 1$.

We know that $K_{i+1, \pi}$ acts faithfully and completely reducibly on $K_i/\Phi(\Gamma/\bF_{i-1}(\Gamma))$. It is clear that we may write $K_i/\Phi(\Gamma/\bF_{i-1}(\Gamma))=V_{i1}+V_{i2}$ where $V_{i1}$ is the $\pi$ part of $K_i/\Phi(\Gamma/\bF_{i-1}(\Gamma))$ and $V_{i2}$ is the $\pi'$ part of $K_i/\Phi(\Gamma/\bF_{i-1}(\Gamma))$ for all $i \geq 1$. Clearly $\bC_{K_{i+1, \pi}}(V_{i1})$ acts faithfully and completely reducibly on $V_{i2}$ and $K_{i+1, \pi}/\bC_{K_{i+1, \pi}}(V_{i1})$ acts faithfully and completely reducibly on $V_{i1}$.


Since the image of $K \cap \bF_2(\Gamma)$ in $K_{2, \pi}$ is abelian. $|(K \cap \bF_2(\Gamma))/ \bF(\Gamma)| \leq b$ and the order of the maximum abelian quotient of $K_{2,\pi}$ is bounded above by $b$ by Theorem ~\ref{pigroupaction}.

Since $L_{3, \pi}=K/\bF_2(\Gamma)$ is abelian. $|L_{3,\pi}| \leq |\bC_{L_{3, \pi}}(V_{21})| \cdot |L_{3, \pi}/\bC_{L_{3, \pi}}(V_{21})| \leq b \cdot b$ by Theorem ~\ref{pigroupaction} and Proposition ~\ref{prop1}.

Next, we show that $|\Gamma : T|_{\pi} \leq b$.

Let $\chi$ be any irreducible character of $G$ lying over $\lambda$. Then every irreducible character of $\Gamma$ that lies over $\chi$ also lies over $\lambda$ and hence has degree divisible by $|\Gamma : T|$. But $\chi$ extends to its stabilizer in $\Gamma$ and thus some irreducible character of $\Gamma$ lying over $\chi$ has degree $\chi(1) |A : C_A(\chi)|$. The $\pi$-part of $|\Gamma : T|$, therefore, divides $|A : \bC_A(\chi)|$, which is at most $b$. 

This gives that $|A| \leq b \cdot b \cdot b \cdot b \leq b^{4}$.
\end{proof}

Since when $(|A|,|G|)=1$, the orbit sizes of $A$ on $\Irr(G)$ are the same as the orbit sizes in the natural action of $A$ on the conjugacy classes of $G$, the following results immediately follow from the previous ones.

\begin{theorem} \label{coprimeactionconj}
Let $A$ be a solvable $\pi$-group that acts faithfully on a solvable $\pi'$-group $G$. Let $b$ be an integer such that $|A : \bC_A(C)| \leq b$ for all $C \in \cl(G)$. Then $|A| \leq b^{49}$.
\end{theorem}

\begin{theorem} \label{coprimeactionconj23}
Let $A$ be a solvable $\pi$-group that acts faithfully on a solvable $\pi'$-group $G$. Assume that $2,3 \not\in \pi$. Let $b$ be an integer such that $|A : \bC_A(C)| \leq b$ for all $C \in \cl(G)$. Then $|A| \leq b^{4}$.
\end{theorem}



\begin{thebibliography}{19}
\bibitem{ASCHGURA} {M. Aschbacher and R. Guralnick}, `On abelian quotient of primitive groups', {Proc. Amer. Math. Soc.} 107 (1989), 89-95.


\bibitem{GLUCK2} {D. Gluck}, `Primes dividing character degrees and character orbit sizes', {Proc. Amer. Math. Soc.} 101 (1987), 219-225.


\bibitem{IMIB} {I.M. Isaacs}, Character theory of finite groups, Dover, New York, 1994.

\bibitem{IMI2} {I.M. Isaacs}, `Large orbits in actions of nilpotent groups', {Proc. Amer. Math. Soc.} 127 (1999), 45-50.

\bibitem {KELL} {T. M. Keller}, `Fixed conjugacy classes of normal subgroups and the $k(GV)$-problem', {Journal of Algebra} 305 (2006), 457-486.

\bibitem{KELLYANG} {T. M. Keller and Y. Yang}, `On abelian quotient and orbit size of solvable linear groups', Preprint (2012).


\bibitem{MAWOLF} {O. Manz and T.R. Wolf}, `Representations of Solvable Groups', {Cambridge University Press}, 1993.

\bibitem {YY1} {Y. Yang}, `Orbits of the actions of finite solvable groups', {Journal of Algebra} 321 (2009), 2012-2021.



\bibitem{YY11} {Y. Yang}, `Blocks of small defect', submitted.


\bibitem{MO1} {A. Moret\'o}, `Large orbits of $p$-groups on characters and applications to character degrees', {Israel J. of Math.} 146 (2005), 243-251.




\end{thebibliography}
\end{document}